\documentclass[12pt, reqno]{amsart}
\usepackage{dcolumn}
\newcolumntype{d}[1]{D{.}{.}{#1}}
\usepackage{amsfonts,amssymb,latexsym,amsmath, amsxtra}
\usepackage{verbatim}
\theoremstyle{plain}
\newtheorem{theorem}{Theorem}[section]
\newtheorem*{theorem*}{Theorem}

\newtheorem{lemma}[theorem]{Lemma}
\newtheorem{proposition}[theorem]{Proposition}
\newtheorem*{conjecture*}{Conjecture}

\theoremstyle{definition}

\newtheorem{example}[theorem]{Example}
\theoremstyle{remark}
\newtheorem{remark}{Remark}

\newtheorem*{remark*}{Remark}
\newtheorem*{example*}{Example}
\newtheorem*{remarks*}{Remarks}

\numberwithin{equation}{section}

\newcommand{\Log}{{\mathrm {Log}}}
\makeatletter
\newcommand*{\rom}[1]{\expandafter\@slowromancap\romannumeral #1@}
\makeatother

\def\({\left(}
\def\){\right)}

\begin{document}
\begin{abstract}
We compute the asymptotic expansion of a function commented on by Ramanujan and show the expansion suggested in his last letter to hardy is asymptotic and non convergent.  This offers some explanation as to why Ramanujan says this series in question ``has an unclosed form at the singularities"
\end{abstract}
\title{Explaining Ramanujan's ``Unclosed" Comment in His Last Letter to Hardy}

\author{Daniel Parry}
\address{Department of Mathematics, University of Cologne}
\email{dan.t.parry@gmail.com}

\maketitle

\section{Introduction}
\let\thefootnote\relax\footnotetext{The author is supported by the European Research Council (ERC) Grant agreement n. 335220 -
AQSER.}
	
In Ramanujan's famous last letter to Hardy, he ponders the definition of what he calls mock theta functions \cite{MR1862750}. In a truly scientific fashion, he communicates a loose definition and lists the well-known 13 positive examples of what he thinks a mock theta function ought to be and at least one negative example of what a mock theta function is not. The positive examples have been extensively studied and have been a part of a vibrant research program in modular forms and number theory \cite{MR2605321}. However, the one particular negative example appears to not be nearly as well studied. This is the function
$$
F(q) =: \sum_{m=0}^\infty \frac{q^{\frac{m(m+1)}{2}}}{(q)_m^2}
$$
with $(z;q)_m=: \prod_{n=0}^{m-1}(1-zq^m),$ $(q)_m =: (q;q)_m$ denote the $q$-pochhammer symbol. Quoting Ramanujan, he states that he believes $F(q)$ has ``an unclosed form at the singularities." Focusing specifically on the $q=1$ singularity, he states the following asymptotic development as $s\to 0,$
$$
F(e^{-s}) = \sqrt{\frac{s}{2\pi\sqrt{5} }}\exp\left(\frac{\pi^2}{5s}+ \frac{1}{8\sqrt{5}}s+c_2s^2+\dots + c_m s^m +O_m\left(s
^{m+1}\right)\right).
$$
For the rest of this paper, we fix $s=-\Log \left(q\right).$
Several prominent mathematicians have provided commentary on Ramanujan's work over the years \cite{MR2744843,MR1862757,MR1573993, MR2135178, MR2474043, MR2952081,MR3113512}. In particular, both Andrews and Watson \cite{MR1862757,MR1573993} have provided commentary on this letter. While Andrews focus is on explaining the difference between theta series, false theta series, and mock theta series, Watson makes some attempt explaining Ramanujan's comment by resorting to a heuristic interpolation argument but nothing rigorous was stated. As of now, there is little rigorous work on this non example.
This work will attempt to quantify
the bad behavior of $F(q)$ at the main root of unity $q\to 1.$ We will do this through a simple observation; $F(q)$ is the dot product of two well known $q$-binomial objects. This leads to a constant term argument which, after unpacking notation, reveals an everywhere divergent power series expansion for $F(e^{-s}).$
\begin{remark}
The method chosen for this paper was the constant term method. However, it was suggested by Larry Rolen that that one might be able to rewrite
$$
F\left(e^{-s}\right) = \sum_{m=0}^\infty \frac{e^{-s m(m+1)/2}}{\Gamma_{e^{-s}}(m+1)^2}(1-e^{-s})^{-2m}  \thicksim \int_{0}^\infty  \frac{e^{-s m(m+1)/2}}{\Gamma(m+1)^2}(1-e^{-s})^{-2m} dm
$$
which allows one to use the Euler Maclaurin summation formula to analyze $F(q)$ as $s\to 0.$ Here $\Gamma_q(z)$ denotes the $q$-gamma function as was studied and approximated by Zhang \cite{MR3128410} (or in a less direct way) McIntosh \cite{MR1703273}. A successful Euler Maclaurin analysis should then yield a proof which is elementary and independent of complex analysis. \end{remark}
\section{Acknowledgments}
I would like to thank Kathrin Bringmann and Larry Rolen for their helpful comments and careful readings on earlier versions of this manuscript and for suggesting the problem.
\section{Statement of Results}
This paper will confirm the comments of Ramanujan in \cite[page 93-94]{MR1862750} if we will take McIntosh's \cite{MR2511670} interpretation of closed. That is, we will prove
\begin{theorem}\label{nonclosedstatement}As $s\to 0,$ 
$$
F(e^{-s}) = \sqrt{\frac{s}{2\pi\sqrt{5} }}\exp\left(\frac{\pi^2}{5s}+ \frac{1}{8\sqrt{5}}s+c_2s^2+\dots + c_m s^m+ O_m\left(s^{m+1}\right)\right)
$$
for infinitely many nonzero $c_j.$ 
\end{theorem}
\begin{remark}
We focus on the main singularity $q\to 1$ but the methods of the paper should be capable of being extended to all roots of unity. We will remark during the proof of where the estimates need to change for this to occur.
\end{remark}
We will prove something slightly stronger. For contradiction, if we suppose that Theorem \ref{nonclosedstatement} were false and then 
$$
F(e^{-s}) = \sqrt{\frac{s}{2\pi\sqrt{5} }}\exp\left(\frac{\pi^2}{5s}\right)R(s)
$$
where the series expansion of $R(s) \thicksim \exp \left(p(s)\right)$ where $p(s)$ is some polynomial in $s.$ This function is an entire function, and the ``down stairs" power series must converge everywhere when $s\in \mathbb{C}.$ This contradicts the main technical theorem of our paper which is listed below.
\begin{theorem}\label{MainTheorem1}
The asymptotic development for $F(q)$ is given by
$$
F(e^{-s}) =\sqrt{\frac{s}{2\pi \sqrt{5}}} e^{\frac{\pi^2}{5s}}\left(1+\frac{1}{8\sqrt{5}}s +b_2s^2+ \dots b_n s^n + O_n\left(s^{n+1}\right)\right).
$$
The series
$$
R(s) =: 1+\sum_{i=1}^\infty b_{i}s^{i}
$$
does not converge.
\end{theorem}
This theorem will follow from two propositions about $R(s).$ The first one is the integral representation of $R(s)$ stated here.  Recall the polylogarithm \cite{MR2290758}
\begin{equation}
\mathrm{Li}_{s}(z) =: \sum_{n=1}^\infty \frac{z^n}{n^s}.
\end{equation} Let $\phi= \frac{1}{2}+\frac{\sqrt{5}}{2}$ and define
\begin{equation}\label{Endefined}
E_{n} =: \mathrm{Li}_{-n}\left(\phi^{-1}\right) -(-1)^{n} \mathrm{Li}_{-n}\left(-\phi\right).
\end{equation}
Using \cite[Page 6-7]{MR2290758} we have the first few $E_k,$ 
\begin{align}E_{-2}&= \textrm{Li}_2\left(\phi^{-1}\right)-\textrm{Li}_2\left(-\phi\right) \\
&= \frac{\pi^2}{10} - \Log^2\left( \phi \right) -\left(- \frac{\pi^2}{10} - \Log^2\left(\phi\right)\right) = \frac{\pi^2}{5}, \label{ETM2}\\
E_{-1}&=\textrm{Li}_1\left(\phi^{-1}\right)+\textrm{Li}_1\left(-\phi\right) =0, \nonumber\\
E_{0} &=\textrm{Li}_0\left(\phi^{-1}\right)-\textrm{Li}_0\left(-\phi\right) =\sqrt{5}, \label{E0} \\ 
E_{1} &=4, \nonumber \\
E_{2} &=8\sqrt{5}. \nonumber
\end{align}
Let $$
B_{n}(x) =: \sum_{j=0}^n {n \choose j}B_{n-j}x^j
$$
denote the Bernoulli polynomials \cite[page 20]{MR1688958}. Let
$$
J(s,v)\thicksim 
\sum_{k=2}^\infty 
E_{k-1} s^{k} 
\frac{
B_{k+1}
\left(
\frac{1}{2}+iv
\right)
}{
(k+1)!
} 
, \qquad s\to 0, \ k\to \infty
$$
denote an asymptotic series and $J_N(s,v)$ denote the truncation of $J\left(s,v\right)$ up to $k=N$ term.
\begin{proposition}\label{mainprop}
For every $N\in \mathbb{N}$ fixed and $s\to 0,$
$$
R(s) = \frac{1}{\sqrt{2\pi}} \int_{-\frac{\sqrt[4]{5}}{\sqrt[6]{s}}}^{\frac{\sqrt[4]{5}}{\sqrt[6]{s}}} e^{-\frac{v^2}{2}}\exp\left(J_{N}\left(s,i\frac{v}{\sqrt{\sqrt{5}s}}\right)-\frac{\sqrt{5}s}{24}\right)dv + O_{N}\left(s^{\frac{N}{2}}\right).
$$
\end{proposition}
Rephrasing Proposition \ref{mainprop}, we have the asymptotic series (as in \cite[Definition 1.1]{MR2238098}) when we take $N \to \infty$ in terms of ascending powers of $s,$
$$
\exp\left(J\left(s,i\frac{v}{\sqrt{\sqrt{5}s}}\right)-\frac{\sqrt{5}s}{24}\right) \thicksim 1+ \sum_{j=1}^\infty a_j(v) s^{\frac{j}{2}}, \quad s\to 0.
$$
Integrating the series term by term, we obtain
\begin{align}\label{unclosednessdef}
R(s)&\thicksim \frac{1}{\sqrt{2\pi}} \int_{-\frac{\sqrt[4]{5}}{\sqrt[6]{s}}}^{\frac{\sqrt[4]{5}}{\sqrt[6]{s}}} e^{-\frac{v^2}{2}}\exp\left(\lim_{N\to \infty}J_N\left(s,i\frac{v}{\sqrt{\sqrt{5}s}}\right)-\frac{\sqrt{5}s}{24}\right)dv \\
&= \frac{1}{\sqrt{2\pi}} \int_{-\frac{\sqrt[4]{5}}{\sqrt[6]{s}}}^{\frac{\sqrt[4]{5}}{\sqrt[6]{s}}} e^{-\frac{v^2}{2}}\exp\left(J\left(s,i\frac{v}{\sqrt{\sqrt{5}s}}\right)-\frac{\sqrt{5}s}{24}\right)dv \nonumber \\
&\thicksim 1+ \sum_{j=1}^\infty s^{j}\frac{1}{\sqrt{2\pi}}\int_{-\frac{\sqrt[4]{5}}{\sqrt[6]{s}}}^{\frac{\sqrt[4]{5}}{\sqrt[6]{s}}} e^{-\frac{v^2}{2}}a_{2j}(v)dv \nonumber \\
&\thicksim 1+ \sum_{j=1}^\infty s^{j}\frac{1}{\sqrt{2\pi}}\int_{-\infty}^{\infty} e^{-\frac{v^2}{2}}a_{2j}\left(v\right)dv=: 1 + \sum_{j=1}^\infty b_j s^j. \label{coeffformula}
\end{align}
Then at the end we will use the divergence of $J$ to force a contradiction. We last have one loose end to tie up. To our best knowledge, $c_1=b_1$ has never been formally computed. So as an example of how to use our theorem, we will compute $b_1.$
\begin{example}
As an example of how to use equation (\ref{coeffformula}), we will verify Ramanujan's leading term $b_1= \frac{1}{8\sqrt{5}}.$ We compute the formal series expansion. Computing the terms that have $s$ degree $\le 1,$ in $j$ we obtain
$$
J\left(s,\frac{iv}{\sqrt{\sqrt{5}s}}\right) = -\frac{iv^3 2 \sqrt[4]{5}}{15} v^3 \sqrt{s}+ \frac{\sqrt{5}}{15}v^4 s +O\left(s^\frac{3}{2}\right).
$$
We then expand the exponential of $J$ and compute the projection
$$[s]\exp \left(J\left(s,\frac{iv}{\sqrt{\sqrt{5}s}}\right)-\frac{s\sqrt{5}}{24}\right) = \left(\frac{1}{15}\sqrt{5}v^4- \frac{\sqrt{5}}{24}\right) + \frac{1}{2}\left(- i v^3 \frac{2}{15}\sqrt[4]{5}\right)^2.
$$
To apply equation (\ref{coeffformula}) we need the two integrals
$$
\frac{1}{\sqrt{2\pi}}\int_{-\infty}^{\infty} e^{-\frac{1}{2}v^2}v^{4}dv = 3, \qquad \frac{1}{\sqrt{2\pi}}\int_{-\infty}^{\infty} e^{-\frac{1}{2}v^2}v^{6}dv = 15.
$$
Hence 
\begin{align*}
R(s)&\thicksim 1+s\left(\frac{\sqrt{5}}{15} \frac{1}{\sqrt{2\pi}}\int_{-\infty}^{\infty} e^{-\frac{1}{2}v^2}v^{4}dv
-\frac{\sqrt{5}}{24}
-\frac{2}{225}\sqrt{5} \frac{1}{\sqrt{2\pi}}\int_{-\infty}^{\infty} e^{-\frac{1}{2}v^2}v^{6}dv 
\right) \\
&\thicksim 1 + s \left(\frac{\sqrt{5}}{5}-\frac{\sqrt{5}}{24}-\frac{2}{15}\sqrt{5}\right) = 1+b_1s +O(s^2)= 1+ \frac{1}{8\sqrt{5}}s+O\left(s^2\right) .
\end{align*}
\end{example}
\section{Proof of Proposition \ref{mainprop}}
The key idea observation for producing asymptotics is that $F(q)$ is the dot product of two sequences
$$
F(q) = \sum_{m=0}^\infty \frac{q^{\frac{m(m+1)}{2}}}{(q)_m^2} =
\left\{\frac{q^{\frac{m(m-1)}{2}} q^{\frac{m}{2}}}{(q)_m}\right\}\Bigg|_{m=0\dots \infty} \cdot  \left\{ \frac{q^\frac{m}{2}}{(q)_m}\right\}\Bigg|_{m=0\dots \infty}.
$$

We can write $F(q)$ as the constant term of the product of well-understood functions. 
For this, define 
\[R(z,q)=:\sum_{m=0}^\infty 
\left(zq^\frac{1}{2}\right)^m 
\frac{q^{\frac{m(m-1)}{2}}}{(q)_m}, 
\quad 
S(z,q) =: \sum_{m=0}^\infty
\left(z q^\frac{1}{2}\right)^{m} \frac{1}{(q)_m}
.\] 
By the $q$-Binomial Theorem \cite[Theorem 10.2.1]{MR1688958},
\[
R(z,q) = \left(-zq^\frac{1}{2};q \right)_\infty , \quad S(z,q) = \frac{1}{\left(zq^\frac{1}{2};q\right)_\infty}.
\]
Using the constant term method, we can write 
\[F(q) = \left[z^{0}\right]S(z,q)R(z^{-1},q) = \left[z^{0}\right]\sum_{m=0}^\infty \frac{z^{m} q^\frac{m}{2}}{\left(q\right)_m}\sum_{n=0}^\infty\frac{\left(-1\right)^n \left(-z\right)^{-n} q^\frac{n^2}{2}}{\left(q\right)_n}.\] 
and we can write this as an integral transform via Cauchy's theorem 
\[F(q) =\frac{1}{2\pi i} \int_{\Gamma} \frac{\left(-z^{-1}q^\frac{1}{2};q\right)_\infty}{\left(zq^\frac{1}{2};q\right)_\infty} \frac{dz}{z}\] 
where $\Gamma$ will be chosen to be a circle of radius $\phi^{-1}.$ Make a variable transformation $z=\phi^{-1}e^{si v}$ with $dz = \phi^{-1}e^{si v} is dv,$ and $z^{-1} = \phi e^{-si v}$
\begin{align*} F(e^{-s}) &=\frac{s}{2\pi }\int_{-\frac{\pi }{s}}^{\frac{\pi }{s}} 
\frac{\left(-\phi e^{-s(iv+\frac{1}{2})};e^{-s}\right)_\infty}{\left(\phi^{-1} e^{-s(-iv+\frac{1}{2})};e^{-s}\right)_\infty}dv\\
& = \frac{s}{2\pi }\left(\int_{|v|>s^{-\frac{2}{3}}}+ \int_{|v|<s^{-\frac{2}{3}}}\right)\frac{\left(-\phi e^{-s(iv+\frac{1}{2})};e^{-s}\right)_\infty}{\left(\phi^{-1} e^{-s(-iv+\frac{1}{2})};e^{-s}\right)_\infty}dv. \end{align*}
We need estimates for the pochhammer symbols so we require a few variants of the theorems in \cite{MR1703273}. The first theorem follow almost verbatim from \cite[page 213]{MR1703273} paying careful attention to the dependence on $v\in \mathbb{R}$. To control the $v$ dependence, one must use the inequality \cite[Page 245]{MR2440898}
\begin{equation}\label{eulereqn}
\left|\mathrm{Li}_{-N}(z)\right| = \left|\frac{A_N(z)}{(1-z)^{N+1}}\right| \le \frac{A_{N}(|z|)}{\left(\left|1-|z|\right|\right)^{N+1}}= \left|Li_{-N}(|z|)\right|.
\end{equation}
on the error on $R_n$ in the proof of that theorem. All other steps should be the same or similar. This holds because $A_n(z)$ is the generating function for Eulerian numbers and hence is a positive (natural number) coefficient polynomial.
\begin{lemma}\label{Liexpansion} For every $v\in \mathbb{R}$ and $w= \phi^{-1},-\phi$ and $N>0,$
$$
\Log \left(\left(w e^{-s\left(\frac{1}{2}\pm iv\right)};e^{-s}\right)_\infty\right) = \sum_{k=-1}^N \mathrm{Li}_{1-k}\left(w e^{-s\left(\frac{1}{2}\pm iv\right)}\right) (-s)^k +O_N\left(s^{N+1}\right).
$$
\end{lemma}
Also we require the \cite[Lemma 2]{MR1703273} in its asymptotic version. 
\begin{lemma} \label{litaylor} For every $w<1,$ $N>0,$ $\ell \in \mathbb{Z},$ and $x\in X\subset \mathbb{C}$ compact
$$
\mathrm{Li}_{-\ell}(we^{x}) = \sum_{n=0}^N \frac{\mathrm{Li}_{-\ell-n}(w)x^n}{n!} + O_{X,N}\left(x^{N+1}\right).
$$
\end{lemma}
\begin{proof}
It follows from the series definition of the polylogarithm
$$
\frac{d}{dx}\mathrm{Li}_{j}\left(we^{x}\right) = \mathrm{Li}_{j-1}\left(we^{x}\right)
$$
and hence the lemma is just a application of Taylor's theorem from calculus. 
\end{proof}
These lemmas allow us control over the ``major" and ``minor" arcs in our computation.
\begin{lemma}\label{coniequality}
For every $s>0$ sufficiently small and $|v|> s^\frac{2}{3},$ 
$$
\left|\frac{\left(-\phi e^{-s(iv+\frac{1}{2})};e^{-s}\right)_\infty}{\left(\phi^{-1} e^{-s(-iv+\frac{1}{2})};e^{-s}
\right)_\infty}\right| \ll \exp\left( \frac{\pi^2}{5s} - \frac{\sqrt{5}}{2s^\frac{1}{3}}\right).
$$
\end{lemma}
\begin{proof}
Define 
$$
h(s,v) =: \mathrm{Re} \left(\mathrm{Li}_2\left(\phi^{-1} e^{-s(-iv+\frac{1}{2})}\right) - \mathrm{Li}_2\left(-\phi e^{-s(iv+\frac{1}{2})} \right)\right).
$$
Apply Lemma \ref{Liexpansion} with $N=-1,$ and obtain
\begin{align} \label{N1Bd}
\left|\frac{\left(-\phi e^{-s(iv+\frac{1}{2})};e^{-s}\right)_\infty}{\left(\phi^{-1} e^{-s(-iv+\frac{1}{2})};e^{-s}\right)_\infty}\right|
\ll \exp\left(\frac{h(s,v)}{s}\right).\end{align}
Taking the derivative of the right hand side we obtain 
\[\frac{\partial }{\partial v}h(s,v)=s\mathrm{Im} \left(\Log\left(\left(1 -\phi^{-1} e^{-s(-iv+\frac{1}{2})}\right)\left(1+\phi e^{-s(iv+\frac{1}{2})}\right)\right)\right)\] 
which can only be zero if $v =0,\pi.$ Inspection shows that for $[0,\pi]$ our function is decreasing and is increasing otherwise. Hence, we can we can bound using Lemma \ref{litaylor} with $N=2,$ and $E_{n}$ defined in equation (\ref{Endefined})
\begin{align} 
\frac{1}{s}h\left(s,v\right) &\le \frac{1}{s}h\left(s,s^{-\frac{2}{3}}\right) \\
& =\frac{1}{s} \mathrm{Re} \left(\mathrm{Li}_2\left(\phi^{-1} e^{is^\frac{1}{3}-\frac{s}{2})}\right) - \mathrm{Li}_2\left(-\phi e^{- is^\frac{1}{3}- \frac{s}{2}}\right)\right) \\
&= \frac{E_{-2}}{s} - \frac{E_{0}}{2s^\frac{1}{3}}+ O(1) \\
&= \frac{\pi^2}{5s} - \frac{\sqrt{5}}{2 s^\frac{1}{3}}+ O(1).\label{inequalitybd}
\end{align}
Substituting inequality (\ref{inequalitybd}) into equation (\ref{N1Bd}) completes the proof.
\end{proof}
We require one last Lemma from \cite{MR1703273}. With the adjustments made by Lemma \ref{litaylor} and Lemma \ref{Liexpansion} this proof follows verbatim from the cited version.
\begin{lemma}\cite[Theorem 3]{MR1703273}\label{pochhamerdefis} For every $s>0,$ and $ w= -\phi,\phi^{-1} <1$ and $v\in \mathbb{R},$
$$
\Log\left( \left(w e^{-s\left(\frac{1}{2}\pm iv\right)},e^{-s}\right)_\infty\right) = \sum_{k=-1}^N \mathrm{Li}_{1-k}(w)\left(-s\right)^{k}\frac{B_{k+1}\left(\frac{1}{2}\pm iv\right)}{(k+1)!} + O\left(s^{N+1}v^{N+2}\right).
$$
\end{lemma}
\begin{remark}
To generalize this work to all roots of unity, all one needs to do is extend the work of McIntosh \cite{MR1703273} to all roots of unity.  It is already known that such an extension exists (see \cite{MR2427666} for the main term for example).  Given the usefulness of \cite{MR1703273} for studying non modular $q$-series, it would be of independent interest to have such a generalization.
\end{remark}
Applying Lemma \ref{pochhamerdefis} and using the identity $B_{k}(z) =(-1)^k B_{k}(1-z),$
\begin{align*}
-\Log \left(\left(\phi^{-1} e^{-s\left(\frac{1}{2}- iv\right)},e^{-s}\right)_\infty\right) &=- \sum_{k=-1}^N \mathrm{Li}_{1-k}(\phi^{-1})\left(-s\right)^{k}\frac{B_{k+1}\left(\frac{1}{2}- iv\right)}{(k+1)!} +I_1 \\
&=\sum_{k=-1}^N \mathrm{Li}_{1-k}\left(\phi^{-1}\right)s^{k}\frac{B_{k+1}\left(\frac{1}{2}+ iv\right)}{(k+1)!} + I_1.
\end{align*}
Where $I_1= O_{N,X}\left(s^{N+1}v^{N+2}\right).$
Directly applying Lemma \ref{pochhamerdefis}, we get
\begin{align*}
\Log \left( \left(-\phi e^{-s\left(\frac{1}{2}+ iv\right)},e^{-s}\right)_\infty \right) &= -\sum_{k=-1}^N (-1)^{1-k}\mathrm{Li}_{1-k}(-\phi) s^k\frac{B_{k+1}\left(\frac{1}{2}+ iv\right)}{(k+1)!} + I_2.
\end{align*}
with $I_2 =O_{N,X}\left(s^{N+1}v^{N+2}\right).$ 
Adding the two terms and using equations (\ref{Endefined}-\ref{E0}) we now calculate
\begin{align*}
\Log\left( \frac{\left(-\phi e^{-s\left(\frac{1}{2}+ iv\right)},e^{-s}\right)_\infty}{\left(\phi^{-1} e^{-s\left(\frac{1}{2}- iv\right)},e^{-s}\right)_\infty} \right) &= \sum_{k=-1}^{N}E_{k-1}s^k \frac{B_{k+1}\left(\frac{1}{2}+iv\right)}{(k+1)!} + O\left(s^{N+1}v^{N+2}\right) \\
&=\frac{E_{-2}}{s} + E_{-1}B_{1}\left(\frac{1}{2}+iv\right) + \frac{E_{0}B_2\left(\frac{1}{2}+iv\right)}{2!}\\ &+ \sum_{k=2}^{N}E_{k-1}s^k \frac{B_{k+1}\left(\frac{1}{2}+iv\right)}{(k+1)!} + O\left(s^{N+1}v^{N+2}\right) \\
&= \frac{\pi^2}{5s} -s\frac{\sqrt{5}}{2}v^2-\frac{\sqrt{5}s}{24} +J_{N}\left(s,v\right) + O_N\left(s^{N+1}v^{N+2}\right).
\end{align*}
We therefore have
\begin{align*}
F(e^{-s}) &= \sqrt{\frac{s}{\sqrt{5}2\pi }}e^{\frac{\pi^2}{5s}} R(s). 
\end{align*}
where
\begin{align*}
R(s)&=: \sqrt{\frac{s\sqrt{5}}{2\pi}} \int_{-s^{-\frac{2}{3}}}^{s^{-\frac{2}{3}}} e^{-\frac{s\sqrt{5}v^2}{2}}\exp \left(J_{N}\left(s,v\right)-\frac{s\sqrt{5}}{24}+ O\left(s^{N+1}v^{N+2}\right)\right)dv+ O_N\left(s^N\right). \\
\end{align*}
Now make a variable transformation $u= \sqrt[4]{5}\sqrt{s}v$ and $u(\pm s^{-\frac{2}{3}}) =\pm \sqrt[4]{5}s^{-\frac{1}{6}},$
\begin{align*}
R(s)&= \frac{1}{\sqrt{2\pi}}\int_{-s^{-\frac{1}{6}}5^{\frac{1}{4}}}^{s^{-\frac{1}{6}}5^{\frac{1}{4}}} e^{-\frac{u^2}{2}}\exp \left(J_{N}\left(s,\frac{u}{\sqrt{\sqrt{5}s}}\right)-\frac{s\sqrt{5}}{24}+ O\left(s^{\frac{N}{2}}u ^{N+2}\right)\right)du+ O_N\left(s^N\right).
\end{align*}
We now require one last lemma to complete the calculation.
\begin{lemma}\label{Jnbound}
For every $N>0,$ and uniformly for $v\in [5^{-\frac{1}{4}}s^{\frac{1}{6}},5^{\frac{1}{4}}s^\frac{1}{6}],$
$$
\lim_{s\to 0} J_{N}\left(s,\frac{u}{\sqrt{\sqrt{5}s}}\right) =0.
$$
\end{lemma}
\begin{proof}
Since $N$
 is fixed and finite, it suffices to show for $k\geq 2,$
$$
s^k\frac{B_{k+1}\left(s,\frac{u}{\sqrt{\sqrt{5}s}}\right)}{(k+1)!} \ll s^{\frac{k-1}{2}} u^k \ll s^{\frac{k}{3}-\frac{1}{2}} \ll s^{\frac{1}{6}}.
$$
Therefore
$$
J_{N}\left(s,\frac{u}{\sqrt{\sqrt{5}s}}\right) = \sum_{k=2}^{N} E_{k-1}s^k\frac{B_{k+1}\left(s,\frac{u}{\sqrt{\sqrt{5}s}}\right)}{(k+1)!} \ll_N s^{\frac{1}{6}}
$$
which completes the proof.
\end{proof}
Lemma \ref{Jnbound} gives us our last required estimate
\begin{align*}
&\frac{1}{\sqrt{2\pi}}\int_{-5^{\frac{1}{4}}s^{-\frac{1}{6}}}^{5^{\frac{1}{4}}s^{-\frac{1}{6}}} e^{-\frac{u^2}{2}}\exp \left(J_{N}\left(s,\frac{u}{\sqrt{\sqrt{5}s}}\right)-\frac{s\sqrt{5}}{24}\right)\left(\exp\left( O\left(s^{\frac{N}{2}}u ^{N+2}\right)\right)-1\right)du \\
&\ll \int_{-5^{\frac{1}{4}}s^{-\frac{1}{6}}}^{5^{\frac{1}{4}}s^{-\frac{1}{6}}} e^{-\frac{u^2}{2}} s^{\frac{N}{2}} u^{N+2}dv \ll_N s^\frac{N}{2}.
\end{align*}
We have
$$
R(s) = \frac{1}{\sqrt{2\pi}}\int_{-5^{\frac{1}{4}}s^{-\frac{1}{6}}}^{5^{\frac{1}{4}}s^{-\frac{1}{6}}} e^{-\frac{v^2}{2}}\exp \left(J_{N}\left(s,\frac{u}{\sqrt{\sqrt{5}s}}\right)-\frac{s\sqrt{5}}{24}\right) + O_N\left(s^\frac{N}{2}\right).
$$
This completes the proposition.
\section{Proof of Theorem \ref{nonclosedstatement}}
We begin with a lemma
\begin{lemma}\label{Ekmetrics}
As $n\to \infty,$
$$
E_n =: \left(\Log\left( \phi \right) \right)^{-n-1} n! \bar{E}_n
$$
where $\bar{E}_n = 1 + O\left(K^n\right)$ for some $K<1.$
\end{lemma}
\begin{proof}
In \cite[equation (13)]{MR2576699} we have
$$
\mathrm{Li}_{-d}(e^{-\mu}) = d! \sum_{n\in \mathbb{Z}}\frac{1}{\left(2\pi i n + \mu\right)^{d+1}}.
$$
Substituting $\mu = \Log\left( \phi\right),\pi i - \Log \left(\phi\right)$ and adding the two terms produces the lemma.
\end{proof}
We also require a proposition.
\begin{proposition}\label{limitthoerem} For $s\in X \subset(0,1)$ compact
$$ \frac{1}{\sqrt{2\pi}} \int_{-\frac{\sqrt[4]{5}}{\sqrt[6]{s}}}^{\frac{\sqrt[4]{5}}{\sqrt[6]{s}}} e^{-\frac{v^2}{2}}\exp\left(\lim_{\ell\to \infty}J_{4\ell+1}\left(s,i\frac{v}{\sqrt{\sqrt{5}s}}\right)-\frac{\sqrt{5}s}{24}\right)dv=0
$$
uniformly in $X.$
\end{proposition}
\begin{proof}
It turns out that the function $J_N(s,v)$ diverges factorially. However, the notation obscures things a bit so to prove Proposition \ref{limitthoerem} we will have to do some notational unpacking. After things get written in a clearer way, the Proposition will be obvious.
For the Bernoulli polynomials we need \cite[Theorem 1]{MR0002378}, and \cite[pg 16,165]{MR1688958}, 
\begin{equation} \label{Bkmetrics}
B_{k}\left(\frac{1}{2}\right) =: 2(2\pi)^{-k}k! \cos\left(\frac{\pi }{2} k \right)\bar{B}_k, \qquad \left|\bar{B}_{k}-1 \right| \ll \left(\frac{1}{2}\right)^{k}.
\end{equation}
We can decompose a term of $J_{N}(s,v)$ using \cite[page 275, Exercise 15]{MR0434929},. 
\begin{equation}\label{berngf}
B_{k+1}\left(\frac{1}{2}+z\right) = \sum_{j=0}^{k+1}{k+1\choose j}B_{k+1-j}\left(\frac{1}{2}\right)z^j.
\end{equation}
To make expressions more concise we change units and let $s= 2\pi \Log \left(\phi\right) \alpha.$ Using Lemma \ref{Ekmetrics}, and equations (\ref{Bkmetrics}), (\ref{berngf})
\begin{align*}
E_{k-1}s^k \frac{B_{k+1}\left(\frac{1}{2} +iv \right)}{(k+1)!} 
&= \bar{E}_{k-1}(k-1)! 
\alpha^{k} \sum_{j=0}^{k+1}\frac{(iv)^j}{j!}(2\pi)^k\frac{ B_{k+1-j}\left(\frac{1}{2}\right)}{(k+1-j)!} \\
&= \frac{1}{\pi }\bar{E}_{k-1}(k-1)! \alpha^k\sum_{n=0}^{k+1}\frac{(2\pi i v )^j}{j!}\bar{B}_{k+1-j}\cos \left(\frac{\pi (k+1-j)}{2}\right).
\end{align*}
We wish to do a bit more work before we obtain our simpler notation. Break up $\cos(z) = \frac{z+\bar{z}}{2}.$
\begin{align*}
E_{k-1}s^k \frac{B_{k+1}\left(\frac{1}{2} +iv \right)}{(k+1)!} 
&= -\frac{1}{2\pi}\bar{E}_{k-1}(k-1)! 
\alpha^k\sum_{j=0}^{k+1}\frac{(2\pi i v)^j\bar{B}_{k+1-j}\left(i^{k+1-j} + i^{j-k-1}\right)}{j!} \\
&= -\frac{1}{2\pi i}\bar{E}_{k-1}(k-1)! 
\alpha^k i^k\sum_{j=0}^{k+1}\frac{(2\pi i v)^j\bar{B}_{k+1-j}\left(i^{-j} +(-1)^{k+1} i^{j}\right)}{j!} \\
&= -\frac{1}{2\pi i}\bar{E}_{k-1}(k-1)! 
\alpha^k i^k\left( e_{k+1,B}\left(2\pi v\right) - (-1)^{k}e_{k+1,B}\left(-2\pi v\right)\right).
\end{align*}
Here we use 
\begin{equation}\label{partialexpdef}
e_{k+1,B}\left(z\right) =: \sum_{j=0}^{k+1} \frac{\bar{B}_{k+1-j}z^{j}}{j!}.
\end{equation}
Note 
\begin{equation}\label{partialexplimit}
\lim_{k\to \infty } e_{k+1,B}\left(z\right) = e^{z}
\end{equation}
uniformly on compact subsets. We then let
$$
T_{k,B}(z) = \frac{1}{2}\left( e_{k+1,B}(2\pi v) - (-1)^{k}e_{k+1,B}(-2\pi v)\right).
$$ 
By summing across $k$ it obtain 
$$
\lim_{\ell \to \infty }\frac{J_{4\ell+1}\left(s,\frac{iv }{\sqrt{\sqrt{5}s}}\right)}{(4\ell)!\alpha^{4\ell+1}} =-\frac{1}{2\pi} \lim_{\ell \to \infty }T_{4\ell+2}\left(v\right) = -\frac{1}{2\pi}\cosh \left(v\right).
$$
Hence, for every $v \neq 0$ we have the point wise limit 
$$
\exp\left(\lim_{\ell\to \infty}J_{4\ell+1}\left(s,\frac{u}{\sqrt{\sqrt{5}s}}\right)\right) =0.
$$
Therefore we have for every $s>0,$
$$ \frac{1}{\sqrt{2\pi}} \int_{-\frac{\sqrt[4]{5}}{\sqrt[6]{s}}}^{\frac{\sqrt[4]{5}}{\sqrt[6]{s}}} e^{-\frac{v^2}{2}}\exp\left(\lim_{\ell\to \infty}J_{4\ell+1}\left(s,i\frac{v}{\sqrt{\sqrt{5}s}}\right)-\frac{\sqrt{5}s}{24}\right)dv=0.
$$
which completes the proof.
\end{proof}

\begin{proof}[Proof of Theorem \ref{nonclosedstatement}]
Suppose $R(s)$ has a convergent series expansion which defines an analytic function. The identity theorem then forces us to match coefficients on equation (\ref{unclosednessdef}) and conclude

\begin{equation}\label{finalequation}
R(s^6)\thicksim \frac{1}{\sqrt{2\pi}}  \int_{-\frac{\sqrt[4]{5}}{s}}^{\frac{\sqrt[4]{5}}{s}} e^{-\frac{v^2}{2}}\exp\left(\lim_{N\to \infty}J_{N}\left(s^6,i\frac{v}{\sqrt{\sqrt{5}}s^3}\right)-\frac{\sqrt{5}s^6}{24}\right)dv
\end{equation}
as $s\to 0.$  This expression is not a well defined function of $s$ as $J_N$ doesn't converge.  To prove this, we assume for contradiction that $J_N$ had a limiting value then Proposition \ref{unclosednessdef} forces $R(s)\thicksim 0.$ Taking the limit as $s\to 0,$ we get $R(0)=0.$ This contradicts Proposition \ref{mainprop} which has established with $N=2,$ that $\lim_{s\to 0}R(s)=1.$ 

Thus as the right hand side of equation \ref{finalequation} isn't well defined, it it cannot represent an analytic function.  
Any power series estimation of $R(s)$ which estimates $R(s)$ as $s\to 0$ must be a divergent power series and must be truncated in $N$ with accuracy up to $O_N\left(s^{\frac{N}{2}}\right).$

Now if we note that the $c_j$ were co finite then $R(s)$ would be asymptotic to an entire function of the form $\exp\left(p(s)\right)$.  Hence the right hand side should have a convergent power series.
\end{proof}

\end{document}